\documentclass[reqno,10pt]{amsart}
\usepackage{amssymb}
\usepackage{amsmath}

\usepackage{amsfonts}

\newtheorem{theorem}{Theorem}
\theoremstyle{plain}
\newtheorem{acknowledgement}{Acknowledgement}

\newtheorem{corollary}{Corollary}

\newtheorem{definition}{Definition}

\newtheorem{remark}{Remark}

\numberwithin{equation}{section}

\input{tcilatex}

\begin{document}
\author{}
\title{}
\maketitle

\begin{center}
\thispagestyle{empty}\textbf{{\Large {Twisted }}}$p$-\textbf{{\Large {adic}}}
$(h,q)$-$L$-\textbf{{\Large {functions}}}

\bigskip

\textbf{{\Large {Yilmaz\ Simsek}}}

\bigskip

\textit{{\Large {*University of Akdeniz, Faculty of Arts and Science,
Department of Mathematics,}}}{\Large {\textit{\ 07058 Antalya}, Turkey}}

{\Large 
}\bigskip 

{\Large \textbf{E-Mail: ysimsek@akdeniz.edu.tr}}

{\Large 
}

\bigskip 

\bigskip

\textbf{{\Large {Abstract}}}
\end{center}

\begin{abstract}
By using $q$-Volkenborn integral on $\mathbb{Z}_{p}$, we (\cite{ysimsek}, 
\cite{simsekCanada}) constructed new generating functions of the $(h,q)$%
-Bernoulli polynomials and numbers. By applying the Mellin transformation to
the generating functions, we constructed integral representation of the
twisted $(h,q)$-Hurwitz function and twisted $(h,q)$-two-variable $L$%
-function. By using these functions, we construct twisted new $(h,q)$%
-partial zeta function which interpolates the twisted $(h,q)$-Bernoulli
polynomials and generalized twisted $(h,q)$-Bernoulli numbers at negative
integers. We give relation between twisted $(h,q)$-partial zeta functions
and twisted $(h,q)$-two-variable $L$-function. We find the value of this
function at $s=0$. We also find residue of this function at $s=1$. We
construct $p$-adic twisted $(h,q)$-$L$-function, which interpolates the
twisted $(h,q)$-Bernoulli polynomials:%
\begin{equation*}
L_{\xi ,p,q}^{(h)}(1-n,t,\chi )=-\frac{B_{n,\chi _{n},\xi }^{(h)}(p^{\ast
}t,q)-\chi _{n}(p)p^{n-1}B_{n,\chi _{n},\xi ^{p}}^{(h)}(p^{-1}p^{\ast
}t,q^{p})}{n}\text{.}
\end{equation*}

\bigskip \bigskip 
\end{abstract}


\bigskip 

\section{Introduction, Definitions and\ Notations}

In \cite{kimpqldiscrmath}, Kim constructed $p$-adic $q$-$L$-functions. He
gave fundamental properties of these functions. By $p$-adic $q$-integral he
also constructed generating function of Carlitz's $q$-Bernoulli number.
Throughout this paper $\mathbb{Z}$, $\mathbb{Z}_{p}$, $\mathbb{Q}_{p}$ and $%
\mathbb{C}_{p}$ will be denoted by the ring of rational integers, the ring
of $p$-adic integers, the field of $p$-adic rational numbers and the
completion of the algebraic closure of $\mathbb{Q}_{p}$, respectively, (see 
\cite{Kim5}, \cite{Kim8}). Let $v_{p}$ be the normalized exponential
valuation of $\mathbb{C}_{p}$ with $\mid p\mid _{p}=p^{-v_{p}(p)}=p^{-1}$.
When one talks of $q$-extension, $q$ is variously considered as an
indeterminate, a complex number $q\in \mathbb{C}$, or $p$-adic number $q\in 
\mathbb{C}_{p}$.\ If $q\in \mathbb{C}_{p}$, then we normally assume $\mid
1-q\mid _{p}<p^{-\frac{1}{p-1}}$, so that $q^{x}=\exp (x\log q)$ for$\mid
x\mid _{p}\leq 1$. If $q\in \mathbb{C}$, then we normally assume$\mid q\mid
<1$.

Kubota and Leopoldt proved the existence of meromorphic functions, $%
L_{p}(s,\chi )$, which is defined over the $p$-adic number field. $%
L_{p}(s,\chi )$ is defined by%
\begin{eqnarray*}
L_{p}(s,\chi ) &=&\sum\limits_{%
\begin{array}{c}
n=1 \\ 
(n,p)=1%
\end{array}%
}^{\infty }\frac{\chi (n)}{n^{s}} \\
&=&(1-\chi (p)p^{-s})L(s,\chi ),
\end{eqnarray*}%
where $L(s,\chi )$ \ is the Dirichlet $L$-function which is defined by%
\begin{equation*}
L(s,\chi )=\sum\limits_{n=1}^{\infty }\frac{\chi (n)}{n^{s}}.
\end{equation*}%
$L_{p}(s,\chi )$ interpolates the values%
\begin{equation*}
L_{p}(1-n,\chi )=\frac{(1-\chi (p)p^{n-1})}{n}B_{n,\chi _{n}}\text{, for }%
n\in \mathbb{Z}^{+}=\left\{ 1,2,3,...\right\} ,
\end{equation*}%
where $B_{n,\chi }$ denotes the $n$th generalized Bernoulli numbers
associate with the primitive Dirichlet character $\chi $, and $\chi
_{n}=\chi w^{-n}$, with the Teichm\"{u}ller character cf. (\cite{B. Ferrero
and R. Greenberg}, \cite{J. Diamond}, \cite{K. Iwasawa}, \cite%
{kimpqldiscrmath}, \cite{kimpqLkyus}, \cite{kimArXiv2005}, \cite%
{tkimnewApropqL-2006}, \cite{N. Koblitz}, \cite{N. Koblitz2}, \cite{Shratani
and S. Yamamoto}, \cite{L. C. Washington}).

Kim, Jang, Rim and Pak\cite{TKim-LCJang-SHRim-Pak} defined twisted $q$%
-Bernoulli numbers by using $p$-adic invariant integrals on $\mathbb{Z}_{p}$%
. They gave twisted $q$-zeta function and $q$-$L$-series which interpolate
twisted $q$-Bernoulli numbers. In \cite{Simsek4}, the author constructed
generating functions of $q$-generalized Euler numbers and polynomials. The
author also constructed a complex analytic twisted $l$-series, which is
interpolated twisted $q$-Euler numbers at non-positive integers. In \cite%
{Simsek1}, \cite{yil1}, the author gave analytic properties of twisted $L$%
-functions. We defined twisted Bernoulli polynomials and numbers. The author
gave the relation between twisted Bernoulli numbers and twisted $L$%
-functions. The author also gave $q$-analogues of these numbers and
functions. Young\cite{young} defined some $p$-adic integral representation
for the two-variable $p$-adic $L$-functions. For powers of the Teichm\"{u}%
ller character, he used the integral representation to extend the $L$%
-function to the large domain, in which it is a meromorphic function in the
first variable and an analytic element in the second. These integral
representations imply systems of congruences for the generalized Bernoulli
polynomials. In \cite{Kim16}, by using $q$-Volkenborn integration, Kim
constructed the new $(h,q)$-extension of the Bernoulli numbers and
polynomials. He defined $(h,q)$-extension of the zeta functions which are
interpolated new $(h,q)$-extension of the Bernoulli numbers and polynomials.
In \cite{ysimsek}, the author define twisted $(h,q)$-Bernoulli numbers, zeta
functions and $L$-function. The author also gave relations between these
functions and numbers. In \cite{kim-rim}, Kim and Rim constructed
two-variable $L$-function, $L(s,x\mid \chi )$. They showed that this
function interpolates the generalized Bernoulli polynomials associated with $%
\chi $. By the Mellin transforms, they gave the complex integral
representation for the two-variable Dirichlet $L$-function. They also found
some properties of the two-variable Dirichlet $L$-function. In \cite{Kim17},
Kim constructed the two-variable $p$-adic $q$-$L$-function which
interpolates the generalized $q$-Bernoulli polynomials associated with
Dirichlet character. He also gave some $p$-adic integrals representation for
this two-variable $p$-adic $q$-$L$-function and derived $q$-extension of the
generalized formula of Diamond and Ferro and Greenberg for the two variable $%
p$-adic $L$-function in terms of the $p$-adic gamma and log gamma function.
In \cite{Simsek-Dkim-SHRim}, Simsek, D. Kim and Rim defined $q$-analogue
two-variable $L$-function. They generalized these functions.

In \cite{tkimnewApropqL-2006}, Kim constructed the new $q$-extension of
generalized Bernoulli polynomials attached to $\chi $ due to his work\cite%
{Kim16} and derived the existence of a specific $p$-adic interpolation
function which interpolate the $q$-extension of generalized Bernoulli
polynomials at negative integers. He gave the values of partial derivative
for this function. In this study, we construct twisted version of Kim's $p$%
-adic $q$-$L$-function.

For $f\in UD(\mathbb{Z}_{p},\mathbb{C}_{p})=\left\{ f\mid f:\mathbb{Z}%
_{p}\rightarrow \mathbb{C}_{p}\text{ is uniformly differentiable function}%
\right\} $, the $p$-adic $q$-integral (or $q$-Volkenborn integration) was
defined by 
\begin{equation}
I_{q}(f)=\int_{\mathbb{Z}_{p}}f(x)d\mu _{q}(x)=\lim_{N\rightarrow \infty }%
\frac{1}{[p^{N}]_{q}}\sum_{x=0}^{p^{N}-1}q^{x}f(x),  \label{Eq-17}
\end{equation}%
where 
\begin{equation*}
\mu _{q}(a+dp^{N}\mathbb{Z}_{p})=\frac{q^{a}}{[dp^{N}]_{q}}\text{, }N\in 
\mathbb{Z}^{+}
\end{equation*}%
and%
\begin{equation*}
\lbrack x]_{q}=\left\{ 
\begin{array}{c}
\frac{1-q^{x}}{1-q}\text{, }q\neq 1 \\ 
x\text{, }q=1%
\end{array}%
\right. \text{cf. (\cite{Kim5},\cite{Kim7},\cite{Kim11},\cite{Kim12},\cite%
{Simsek-Dkim-SHRim}).}
\end{equation*}

\begin{equation}
I_{1}(f)=\lim_{q\rightarrow 1}I_{q}(f)=\int_{\mathbb{Z}_{p}}f(x)d\mu
_{1}(x)=\lim_{N\rightarrow \infty }\frac{1}{p^{N}}\sum_{x=0}^{p^{N}-1}f(x)
\label{Equ-18}
\end{equation}%
cf. (\cite{Kim5}, \cite{Kim11}).

If we take $f_{1}(x)=f(x+1)$ in (\ref{Equ-18}), then we have%
\begin{equation}
I_{1}(f_{1})=I_{1}(f)+f^{^{\prime }}(0),  \label{Equ-19}
\end{equation}%
where $f^{^{\prime }}(0)=\frac{d}{dx}f(x)\mid _{x=0},$ cf. (\cite{Kim14}, 
\cite{Kim8}).

Let $p$ be a fixed prime. For a fixed positive integer $f$ with $(p,f)=1$,
we set (see \cite{Kim5})

\begin{eqnarray*}
\mathbb{X} &=&\mathbb{X}_{f}=\lim_{\overset{\leftarrow }{N}}\mathbb{Z}/fp^{N}%
\mathbb{Z}, \\
\text{ }\mathbb{X}_{1} &=&\mathbb{Z}_{p}, \\
\text{ }\mathbb{X}^{\ast } &=&\underset{_{%
\begin{array}{c}
0<a<fp \\ 
(a,p)=1%
\end{array}%
}}{\cup }a+fp\mathbb{Z}_{p}
\end{eqnarray*}%
and%
\begin{equation*}
a+fp^{N}\mathbb{Z}_{p}=\left\{ x\in \mathbb{X}\mid x\equiv a(\func{mod}%
fp^{N})\right\} ,
\end{equation*}%
where $a\in \mathbb{Z}$ satisfies the condition $0\leq a<fp^{N}$. For $f\in
UD(\mathbb{Z}_{p},\mathbb{C}_{p})$, 
\begin{equation}
\int_{\mathbb{Z}_{p}}f(x)d\mu _{1}(x)=\int_{\mathbb{X}}f(x)d\mu _{1}(x),
\label{a0}
\end{equation}%
(see \cite{Kim7}, \cite{Kim12}, for details). By (\ref{Equ-19}), we easily
see that%
\begin{equation}
I_{1}(f_{b})=I_{1}(f)+\sum_{j=0}^{b-1}f^{^{\prime }}(j),  \label{Equ-11}
\end{equation}%
where $f_{b}(x)=f(x+b),$ $b\in \mathbb{Z}^{+}$.

Let 
\begin{equation*}
T_{p}=\cup _{n\geq 1}C_{p^{n}}=\lim_{n\rightarrow \infty }C_{p^{n}},
\end{equation*}%
where $C_{p^{n}}=\left\{ \xi \in \mathbb{C}_{p}\mid \xi ^{p^{n}}=1\right\} $
is the cyclic group of order $p^{n}$. For $\xi \in T_{p}$, we denote by $%
\phi _{\xi }:\mathbb{Z}_{p}\rightarrow \mathbb{C}_{p}$ the locally constant
function $x\rightarrow \xi ^{x}$(\cite{kimTWISTber}, \cite%
{TKim-LCJang-SHRim-Pak}).

By using $q$-Volkenborn integration (\cite{Kim5}, \cite{Kim7}, \cite{Kim8}, 
\cite{Kim11}, \cite{Kim12}, \cite{Kim16}), the author\cite{ysimsek}
constructed generating function of the twisted $(h,q)$-extension of
Bernoulli numbers $B_{n,\xi }^{(h)}(q)$ by means of the following generating
function%
\begin{equation*}
F_{\xi ,q}^{(h)}(t)=\frac{\log q^{h}+t}{\xi q^{h}e^{t}-1}=\sum_{n=0}^{\infty
}B_{n,\xi }^{(h)}(q)\frac{t^{n}}{n!}.
\end{equation*}%
By using the above equation, and following the usual convention of
symbolically replacing $(B_{\xi }^{(h)}(q))^{n}$ by $B_{n,\xi }^{(h)}(q)$,
we have%
\begin{eqnarray}
B_{0,\xi }^{(h)}(q) &=&\frac{\log q^{h}}{\xi q^{h}-1}  \label{ayb1} \\
\xi q^{h}(B_{\xi }^{(h)}(q)+1)^{n}-B_{n,\xi }^{(h)}(q) &=&\delta _{1,n},%
\text{ }n\geq 1,  \notag
\end{eqnarray}%
where $\delta _{1,n}$ is denoted Kronecker symbol. We note that if $\xi
\rightarrow 1$, then $B_{n,\xi }^{(h)}(q)\rightarrow B_{n}^{(h)}(q)$ and $%
F_{\xi ,q}^{(h)}(t)\rightarrow F_{q}^{(h)}(t)=\frac{h\log q+t}{q^{h}e^{t}-1}$
(see \cite{Kim16}). If $\xi \rightarrow 1$ and $q\rightarrow 1$, then $%
F_{\xi ,q}(t)\rightarrow F(t)=\frac{t}{e^{t}-1}$\ and $B_{n,\xi
}(q)\rightarrow B_{n}$ are the usual Bernoulli numbers (see \cite%
{Kim-Simsek-Srivastava}).

\begin{remark}
Shiratani and Yamamoto\cite{Shratani and S. Yamamoto} constructed a $p$-adic
interpolation $G_{p}(s,u)$ of the Frobenius-Euler numbers $H_{n}(u)$ and as
its application, they obtained an explicit formula for $L_{p}^{^{\prime
}}(0,\chi )$ with any Dirichlet character $\chi $. Let $u$ be an algebraic
number. For $u\in \mathbb{C}$ with $|u|>1$, the Frobenius-Euler numbers $%
H_{n}(u)$ belonging to $u$ are defined by means of of the generating
function 
\begin{equation*}
\frac{1-u}{e^{t}-u}=e^{H(u)t}
\end{equation*}%
with usual convention of symbolically replacing $H^{n}(u)$ by $H_{n}(u)$.
Thus for the Frobenius-Euler numbers $H_{n}(u)$ belonging to $u$, we have (
see\cite{K. Shiratani}) 
\begin{equation*}
\frac{1-u}{e^{t}-u}=\sum_{n=0}^{\infty }H_{n}(u)\frac{t^{n}}{n!}.
\end{equation*}%
By using the above equation, and following the usual convention of
symbolically replacing $H^{n}(u)$ by $H_{n}(u)$, we have%
\begin{equation*}
\text{ }H_{0}=1\text{ and }(H(u)+1)^{n}=uH_{n}(u)\text{ for (}n\geq 1\text{).%
}
\end{equation*}%
We also note that 
\begin{equation*}
H_{n}(-1)=\mathfrak{E}_{n},
\end{equation*}%
where $\mathfrak{E}_{n}$\ denotes the aforementioned Tsumura version (see%
\cite{K. Shiratani}) of the classical Euler numbers $E_{n}$ which we
recalled above. Let $\xi ^{r}=1,$ $\xi \neq 1$.%
\begin{equation*}
\frac{t}{\xi e^{t}-1}=\sum_{n=0}^{\infty }B_{n,\xi }\frac{t^{n}}{n!}\text{
cf. \cite{kimTWISTber}}
\end{equation*}%
\begin{eqnarray*}
\frac{t}{\xi e^{t}-1} &=&t\frac{1-\xi ^{-1}}{e^{t}-\xi ^{-1}} \\
&=&\frac{t}{\xi -1}\frac{1-\xi ^{-1}}{e^{t}-\xi ^{-1}} \\
&=&\frac{1}{\xi -1}\sum_{n=0}^{\infty }(n+1)H_{n}(\xi ^{-1})\frac{t^{n+1}}{%
(n+1)!}
\end{eqnarray*}%
By comparing the coefficients on both sides of the above equations, we
easily see that 
\begin{equation*}
B_{n+1,\xi }=\frac{1}{\xi -1}(n+1)H_{n}(\xi ^{-1}).
\end{equation*}%
Therefore, if $\xi \neq 1$, then we obtain relations between
Frobenius-Eulernumbers, $H_{n}(\xi ^{-1})$\ and twisted Bernoulli numbers, $%
B_{n,\xi }$. If $\xi =1$, then twisted Bernoulli numbers, $B_{n,\xi }$ are
reduced to classical Bernoulli numbers, $B_{n}$, for detail about this
numbers and polynomials see cf. (\cite{kimTWISTber}, \cite{Kim11}, \cite%
{Kim12}, \cite{Kim14}, \cite{Kim17}, \cite{Kim5}, \cite{Kim7}, \cite{Kim8}, 
\cite{TKim-LCJang-SHRim-Pak}, \cite{N. Koblitz}, \cite{N. Koblitz2}, \cite%
{K. Shiratani}, \cite{Shratani and S. Yamamoto}, \cite{Simsek1}, \cite%
{Simsek4}, \cite{yil1}, \cite{ysimsek}, \cite{ysimsekjmaa-2006}, \cite%
{Kim-Simsek-Srivastava}, \cite{cenkci}).
\end{remark}

Twisted $(h,q)$-extension of Bernoulli polynomials $B_{n,\xi }^{(h)}(z,q)$
are defined by means of the generating function\cite{ysimsek}%
\begin{equation}
F_{\xi ,q}^{(h)}(t,z)=\frac{(t+\log q^{h})e^{tz}}{\xi q^{h}e^{t}-1}%
=\sum_{n=0}^{\infty }B_{n,\xi }^{(h)}(z,q)\frac{t^{n}}{n!},  \label{a2}
\end{equation}%
where $B_{n,\xi }^{(h)}(0,q)=B_{n,\xi }^{(h)}(q)$. By using Cauchy product
in (\ref{a2}), we have%
\begin{equation}
B_{n,\xi }^{(h)}(z,q)=\sum_{k=0}^{n}\left( 
\begin{array}{c}
n \\ 
k%
\end{array}%
\right) z^{n-k}B_{k,\xi }^{(h)}(q).  \label{AA-1}
\end{equation}

We summarize our paper as follows:

In section 2, by applying the Mellin transformation to the generating
functions of the Bernoulli polynomials and generalized Bernoulli
polynomials, we give integral representation of the twisted $(h,q)$-Hurwitz
function and twisted $(h,q)$-two-variable $L$-function. By using these
functions, we construct twisted new $(h,q)$-partial zeta function which
interpolates the twisted $(h,q)$-Bernoulli polynomials at negative integers.
We give relation between twisted $(h,q)$-partial zeta functions and twisted $%
(h,q)$-two-variable $L$-function.

In section 3, we construct $p$-adic twisted $(h,q)$-$L$-functions ($L_{\xi
,p,q}^{(h)}(s,t,\chi )$), which are interpolate the twisted generalized $%
(h,q)$-Bernoulli polynomials \ at negative integers. We calculate residue of 
$L_{\xi ,p,q}^{(h)}(s,t,\chi )$ at $s=1$. We also give fundamental
properties of this functions.

\section{$(h,q)$-partial zeta functions}

Our primary aim in this section is to define twisted $(h,q)$-partial zeta
functions. We give the relation between generating function in (\ref{a2})
and twisted $(h,q)$-Hurwitz zeta function\cite{simsekCanada}. In this
section, we assume that $q\in \mathbb{C}$ with $\mid q\mid <1$. For $s\in 
\mathbb{C}$, by applying the Mellin transformation to (\ref{a2}), we have%
\begin{equation*}
\frac{1}{\Gamma (s)}\int_{0}^{\infty }t^{s-2}F_{\xi ,q}^{(h)}(-t,x)dt=\zeta
_{\xi ,q}^{(h)}(s,x).
\end{equation*}%
By using the above equation, we\cite{ysimsek}, \cite{simsekCanada} defined
twisted $(h,q)$-Hurwitz zeta function as follows:

\begin{definition}
Let $s\in \mathbb{C}$, $x\in \mathbb{R}^{+}$. We define 
\begin{equation}
\zeta _{\xi ,q}^{(h)}(s,x)=\sum_{n=0}^{\infty }\frac{\xi ^{n}q^{nh}}{%
(n+x)^{s}}-\frac{h\log q}{s-1}\sum_{n=0}^{\infty }\frac{\xi ^{n}q^{nh}}{%
(n+x)^{s-1}}.  \label{Equ-14ii}
\end{equation}
\end{definition}

\begin{remark}
Observe that when $w\rightarrow 1$, $\zeta _{\xi ,q}^{(h)}(s,x)$ reduces to%
\begin{equation*}
\zeta _{q}^{(h)}(s,x)=\sum_{n=0}^{\infty }\frac{q^{nh}}{(n+x)^{s}}-\frac{%
h\log q}{s-1}\sum_{n=0}^{\infty }\frac{q^{nh}}{(n+x)^{s-1}}
\end{equation*}%
(see \cite{Kim16}). When $q\rightarrow 1$, $\xi \rightarrow 1$, $\zeta
_{q}^{(h)}(s)$ reduces to $\zeta (s)=\sum_{n=1}^{\infty }\frac{1}{n^{s}}$,
Riemann zeta function and $\zeta _{q}^{(h)}(s,x)$ reduces to $\zeta
(s,x)=\sum_{n=1}^{\infty }\frac{1}{(n+x)^{s}}$, Hurwitz zeta function.
Observe that when $x=1$ in (\ref{Equ-14ii}), we easily see that $\zeta _{\xi
,q}^{(h)}(s,1)=\zeta _{\xi ,q}^{(h)}(s)$, which denotes twisted zeta
function (see \cite{ysimsek}). We also note that $\zeta _{\xi ,q}^{(h)}(s)$
are analytically continued for $Re(s)>1$. $\lim_{\xi \rightarrow 1}\zeta
_{\xi ,q}^{(h)}(s)=\zeta _{q}^{(h)}(s)$, which is given in \cite{Kim16}.
\end{remark}

\begin{theorem}
\cite{simsekCanada}\label{T-3}Let $n\in \mathbb{Z}^{+}$. We obtain%
\begin{equation}
\zeta _{\xi ,q}^{(h)}(1-n,x)=-\frac{B_{n,\xi }^{(h)}(x,q)}{n}.  \label{aa7}
\end{equation}
\end{theorem}

Twisted $(h,q)$-$L$-function is defined as follows:

\begin{definition}
(\cite{ysimsek}) Let $s\in \mathbb{C}$. Let $\chi $ be a Dirichlet character
of conductor $f\in \mathbb{Z}^{+}$. We define%
\begin{equation}
L_{\xi ,q}^{(h)}(s,\chi )=\sum_{n=1}^{\infty }\frac{q^{nh}\xi ^{n}\chi (n)}{%
n^{s}}-\frac{\log q^{h}}{s-1}\sum_{n=1}^{\infty }\frac{q^{nh}\xi ^{n}\chi (n)%
}{n^{s-1}}.  \label{Equ-14iii}
\end{equation}
\end{definition}

Observe that if $\xi \rightarrow 1,$ (\ref{Equ-14iii}) reduces to $%
L_{q}^{(h)}(s,\chi )$ function (see \cite{Kim16}).

\begin{theorem}
\label{T-4}(\cite{ysimsek}) Let $\chi $ be a Dirichlet character of
conductor $f\in \mathbb{Z}^{+}$. Let $n\in \mathbb{Z}^{+}$. We have 
\begin{equation*}
L_{\xi ,q}^{(h)}(-n,\chi )=-\frac{B_{n+1,\chi ,\xi }^{(h)}(q)}{n+1}.
\end{equation*}
\end{theorem}

Relation between $\zeta _{w,q}^{(h)}(s,z)$ and $L_{w,q}^{(h)}(s,\chi )$ are
given by the following theorem\cite{ysimsek}:

\begin{theorem}
\label{t-12}Let $s\in \mathbb{C}$. Let $\chi $ be a Dirichlet character of
conductor $f\in \mathbb{Z}^{+}$. We have%
\begin{equation}
L_{\xi ,q}^{(h)}(s,\chi )=\frac{1}{f^{s}}\sum_{a=0}^{f-1}q^{ha}\xi ^{a}\chi
(a)\zeta _{\xi ^{f},q^{f}}^{(h)}(s,\frac{a}{f}).  \label{qlhzeta}
\end{equation}
\end{theorem}

The generalized twisted $(h,q)$-extension of Bernoulli polynomials $%
B_{n,\chi ,\xi }^{(h)}(z,q)$ are defined by means of the generating function%
\cite{simsekCanada}:%
\begin{eqnarray}
F_{\chi ,w,q}^{(h)}(t,z) &=&t\sum_{a=1}^{f}\frac{\chi (a)\phi _{\xi
}(a)q^{ha}e^{(z+a)t}}{\xi ^{f}q^{hf}e^{ft}-1}+\log q^{h}\sum_{a=1}^{f}\frac{%
\chi (a)\phi _{\xi }(a)q^{ha}e^{(z+a)t}}{\xi ^{f}q^{hf}e^{ft}-1}  \notag \\
&=&\sum_{a=1}^{f}\frac{\chi (a)\phi _{\xi }(a)q^{ha}e^{(z+a)t}(t+\log q^{h})%
}{\xi ^{f}q^{hf}e^{ft}-1}  \label{ay1} \\
&=&\sum_{n=0}^{\infty }B_{n,\chi ,\xi }^{(h)}(z,q)\frac{t^{n}}{n!}\text{,
cf. (\cite{ysimsek}, \cite{simsekCanada}).}  \notag
\end{eqnarray}%
By (\ref{ay1}), we now give the following relation:%
\begin{equation*}
F_{\chi ,w,q}^{(h)}(t,z)=-t\sum_{m=0}^{\infty }\chi (m)\xi
^{m}q^{hm}e^{(z+m)t}-\log q^{h}\sum_{m=0}^{\infty }\chi (m)\xi
^{m}q^{hm}e^{(z+m)t}.
\end{equation*}%
The series on the right-hand side of the above equations is uniformly
convergent. Thus we obtain%
\begin{eqnarray*}
B_{k,\chi ,\xi }^{(h)}(z,q) &=&\frac{d^{k}}{dt^{k}}F_{\chi
,w,q}^{(h)}(t,z)\mid _{t=0} \\
&=&-k\sum_{m=0}^{\infty }\chi (m)\xi ^{m}q^{hm}(m+z)^{k-1}-\log
q^{h}\sum_{m=0}^{\infty }\chi (m)\xi ^{m}q^{hm}(m+z)^{k}
\end{eqnarray*}%
After some elementary calculations, we have%
\begin{eqnarray}
&&\sum_{m=0}^{\infty }\chi (m)\xi ^{m}q^{hm}(m+z)^{k-1}+\frac{\log q^{h}}{k}%
\sum_{m=0}^{\infty }\chi (m)\xi ^{m}q^{hm}(m+z)^{k}  \label{a99} \\
&=&-\frac{B_{k,\chi ,\xi }^{(h)}(z,q)}{k},  \notag
\end{eqnarray}%
where%
\begin{equation*}
B_{n,\chi ,\xi }^{(h)}(z,q)=\sum_{k=0}^{n}\left( 
\begin{array}{c}
n \\ 
k%
\end{array}%
\right) z^{n-k}B_{k,\chi ,\xi }^{(h)}(q).
\end{equation*}%
For any positive integer $n$, we have%
\begin{equation}
B_{n,\chi ,\xi }^{(h)}(z,q)=f^{n-1}\sum_{a=1}^{f}\chi (a)\phi _{\xi
}(a)q^{ha}B_{n,\xi ^{f}}^{(h)}(\frac{a+z}{f},q^{f}).  \label{aa6}
\end{equation}

Note that $B_{n,\chi ,\xi }^{(h)}(0,q)=B_{n,\chi ,\xi }^{(h)}(q),$ $%
\lim_{q\rightarrow 1}B_{n,\chi ,\xi }^{(h)}(q)=B_{n,\chi ,\xi }^{(h)}$,
where $B_{n,\chi ,\xi }^{(h)}$ are \ the twisted Bernoulli numbers (see \cite%
{ysimsek}). In \cite{kimTWISTber}, Kim studied analogue of Bernoulli
numbers, which is called twisted Bernoulli numbers in this paper. He proved
relation between these numbers and Frobenious-Euler numbers. If $\xi
\rightarrow 1$ and $q\rightarrow 1$, then $B_{n,\chi ,\xi }(q)\rightarrow
B_{n,\chi }$ are the usual generalized Bernoulli numbers, and $B_{n,\chi
,\xi }(z,q)\rightarrow B_{n,\chi }(z)$ are the usual generalized Bernoulli
polynomials (see \cite{Kim-Simsek-Srivastava}). If $\xi \rightarrow 1,$ then 
$F_{\xi ,q}^{(h)}(t,z)\rightarrow F_{q}^{(h)}(t,z)$ (see \cite{Kim16}).

By (\ref{ay1}), we define new twisted two-variable $(h,q)$-$L$-functions.
For $s\in \mathbb{C}$, we consider the below integral which is known the
Mellin transformation of $F_{\chi ,\xi ,q}^{(h)}(t,z)$\cite{simsekCanada}.

\begin{equation}
\frac{1}{\Gamma (s)}\int_{0}^{\infty }t^{s-2}F_{\chi ,\xi
,q}^{(h)}(-t,z)dt=L_{\xi ,q}^{(h)}(s,z,\chi ).  \label{a98}
\end{equation}

We are now ready to define the new twisted two-variable $(h,q)$-$L$
function. By using the above integral representation we arrive at the
following definition:

\begin{definition}
\cite{simsekCanada}Let $s\in \mathbb{C}$. Let $\chi $ be a Dirichlet
character of conductor $f\in \mathbb{Z}^{+}$. We define\bigskip 
\begin{equation}
L_{\xi ,q}^{(h)}(s,z,\chi )=\sum_{m=0}^{\infty }\frac{\chi (m)\phi _{\xi
}(m)q^{hm}}{(z+m)^{s}}-\frac{\log q^{h}}{s-1}\sum_{m=0}^{\infty }\frac{\chi
(m)\phi _{\xi }(m)q^{hm}}{(z+m)^{s-1}}.  \label{ay2}
\end{equation}
\end{definition}

Relation between $\zeta _{\xi ,q}^{(h)}(s,z)$ and $L_{\xi ,q}^{(h)}(s,z,\chi
)$ is given by the following theorem:

\begin{theorem}
\cite{simsekCanada}\label{T-5}We have%
\begin{equation}
L_{\xi ,q}^{(h)}(s,z,\chi )=\frac{1}{f^{s}}\sum_{a=1}^{f}q^{ha}\xi ^{a}\chi
(a)\zeta _{\xi ^{f},q^{f}}^{(h)}(s,\frac{a+z}{f}).  \label{aa5}
\end{equation}
\end{theorem}

\begin{theorem}
\label{T-6}Let $\chi $ be a Dirichlet character of conductor $f\in \mathbb{Z}%
^{+}$. Let $n\in \mathbb{Z}^{+}$. We have 
\begin{equation}
L_{\xi ,q}^{(h)}(1-n,z,\chi )=-\frac{B_{n,\chi ,\xi }^{(h)}(z,q)}{n}.
\label{aa6a}
\end{equation}
\end{theorem}

\begin{proof}
Substituting $s=1-n$, $n\in \mathbb{Z}^{+}$ into (\ref{ay2}), we obtain%
\begin{equation*}
L_{\xi ,q}^{(h)}(1-n,z,\chi )=\sum_{m=0}^{\infty }\frac{\chi (m)\phi _{\xi
}(m)q^{hm}}{(z+m)^{1-n}}-\frac{\log q^{h}}{s-1}\sum_{m=0}^{\infty }\frac{%
\chi (m)\phi _{\xi }(m)q^{hm}}{(z+m)^{-n}}.
\end{equation*}
Substituting (\ref{a99}) into the above equation, we arrive at the desired
result.
\end{proof}

\begin{remark}
Note that Proof of (\ref{aa6a}) runs parallel to that of Theorem 8 in \cite%
{Kim-Simsek-Srivastava}, for $s=1-n,$ $n\in \mathbb{Z}^{+}$ and by using
Cauchy Residue Theorem in (\ref{a98}), we arrive at the another proof the
above theorem\cite{simsekCanada}. Observe that $\lim_{\xi \rightarrow
1}L_{\xi ,q}^{(h)}(s,1,\chi )=L_{q}^{(h)}(s,\chi )$. For $q\rightarrow 1$
and $z=1$, then relations (\ref{ay2}) reduces to the following well-known
definition:

Let $r$ $\in \mathbb{Z}^{+},$ set of positive integers, let $\chi $ be a
Dirichlet character of conductor $f\in \mathbb{Z}^{+}$, and let $\xi ^{r}=1$%
, $\xi \neq 1$. Twisted $L$-functions are defined by \cite{N. Koblitz2}%
\begin{equation*}
L_{\xi }(s,\chi )=\sum\limits_{n=1}^{\infty }\frac{\chi (n)\xi ^{n}}{n^{s}}.
\end{equation*}%
Since the function $n\rightarrow \chi (n)\xi ^{n}$ has period $fr,$ this is
a special case of the Dirichlet $L$-functions. Koblitz\cite{N. Koblitz2} and
the author gave relation between $L(s,\chi ,\xi )$\ and twisted Bernoulli
numbers, $B_{n,\chi ,\xi }$ at non-positive integers(see \cite{N. Koblitz}, 
\cite{N. Koblitz2}, \cite{Simsek1}, \cite{yil1}).
\end{remark}

Let $s$ be a complex variable, $a$ and $f$ be integers with $0<a<f$. Then we
define new twisted $(h,q)$-partial zeta function as follows:

\begin{definition}
\begin{equation*}
H_{\xi ,q}^{(h)}(s,a:f)=\sum_{%
\begin{array}{c}
n\equiv a(\func{mod}f) \\ 
n>0%
\end{array}%
}^{\infty }\frac{q^{nh}\xi ^{n}}{n^{s}}-\frac{\log q^{h}}{s-1}\sum_{%
\begin{array}{c}
n\equiv a(\func{mod}f) \\ 
n>0%
\end{array}%
}^{\infty }\frac{q^{nh}\xi ^{n}}{n^{s-1}}.
\end{equation*}
\end{definition}

By using the above definition, relation between $H_{\xi ,q}^{(h)}(s,a:f)$
and $\zeta _{\xi ,q}^{(h)}(s,x)$ are given by

\begin{theorem}
\begin{equation}
H_{\xi ,q}^{(h)}(s,a:f)=q^{ha}\xi ^{a}f^{-s}\zeta _{\xi ^{f},q^{f}}^{(h)}(s,%
\frac{a}{f}).  \label{aa7pH}
\end{equation}
\end{theorem}

\begin{remark}
The function $H_{\xi ,q}^{(h)}(s,a:f)$ is meromorphic function for $s\in 
\mathbb{C}$ with simple pole at $s=1$, having residue, $\func{Re}%
z_{s=1}H_{\xi ,q}^{(h)}(s,a:f)$:%
\begin{eqnarray*}
\func{Re}z_{s=1}H_{\xi ,q}^{(h)}(s,a &:&f)=\lim_{s\rightarrow 1}(s-1)H_{\xi
,q}^{(h)}(s,a:f) \\
&=&\frac{q^{ha}\xi ^{a}\log q^{h}}{q^{hf}\xi ^{f}-1}.
\end{eqnarray*}
\end{remark}

By (\ref{aa7}) and (\ref{aa7pH}), we have

\begin{corollary}
Let $n\in \mathbb{Z}^{+}$. We have%
\begin{equation}
H_{\xi ,q}^{(h)}(1-n,a:f)=-\frac{q^{ha}\xi ^{a}f^{n-1}B_{n,\xi ^{f}}^{(h)}(%
\frac{a}{f},q^{f})}{n}.  \label{hq-1}
\end{equation}
\end{corollary}

We modify the twisted $(h,q)$-extension of the partial zeta function as
follows:

\begin{corollary}
Let $s\in \mathbb{C}$. We have%
\begin{equation}
H_{\xi ,q}^{(h)}(s,a:f)=\frac{a^{s-1}q^{ha}\xi ^{a}}{(s-1)f}%
\sum_{k=0}^{\infty }\left( 
\begin{array}{c}
1-s \\ 
k%
\end{array}%
\right) \left( \frac{f}{a}\right) ^{k}B_{k,\xi ^{f}}^{(h)}(q^{f}).
\label{hq2}
\end{equation}
\end{corollary}

\begin{proof}
By using (\ref{AA-1}) and (\ref{hq-1}), we have%
\begin{equation*}
H_{\xi ,q}^{(h)}(1-n,a:f)=-\frac{q^{ha}\xi ^{a}f^{n-1}}{n}%
\sum_{k=0}^{n}\left( 
\begin{array}{c}
n \\ 
k%
\end{array}%
\right) \left( \frac{a}{f}\right) ^{n-k}B_{k,\xi ^{f}}^{(h)}(\frac{a}{f}%
,q^{f}).
\end{equation*}%
Substituting $s=1-n$, and after some elementary calculations, we arrive at
the desired result.
\end{proof}

Observe that if $\xi =1$, then $H_{q}^{(h)}(s,a:f)$ is reduced to the
following equation cf. \cite{tkimnewApropqL-2006}:%
\begin{equation*}
H_{q}^{(h)}(s,a:f)=\frac{a^{s-1}q^{ha}}{(s-1)f}\sum_{k=0}^{\infty }\left( 
\begin{array}{c}
1-s \\ 
k%
\end{array}%
\right) \left( \frac{f}{a}\right) ^{k}B_{k}^{(h)}(q^{f}).
\end{equation*}

By using (\ref{qlhzeta}), (\ref{aa7pH}) and (\ref{hq2}), we arrive at the
following theorem:

\begin{theorem}
Let $s\in \mathbb{C}$. Let $\chi $ ($\chi $ $\neq 1$) be a Dirichlet
character of conductor $f\in \mathbb{Z}^{+}$.%
\begin{eqnarray*}
L_{\xi ,q}^{(h)}(s,\chi ) &=&\sum_{a=1}^{f}\chi (a)H_{\xi ,q}^{(h)}(s,a:f) \\
&=&\frac{1}{(s-1)f}\sum_{a=1}^{f}\chi (a)a^{s-1}q^{ha}\xi
^{a}\sum_{k=0}^{\infty }\left( 
\begin{array}{c}
1-s \\ 
k%
\end{array}%
\right) \left( \frac{f}{a}\right) ^{k}B_{k,\xi ^{f}}^{(h)}(q^{f}).
\end{eqnarray*}
\end{theorem}

We now define new twisted $(h,q)$-partial Hurwitz zeta function as follows:

\begin{definition}
\begin{equation*}
H_{\xi ,q}^{(h)}(s,x+a:f)=\sum_{%
\begin{array}{c}
n\equiv a(\func{mod}f) \\ 
n\geq 0%
\end{array}%
}^{\infty }\frac{q^{nh}\xi ^{n}}{(x+n)^{s}}-\frac{\log q^{h}}{s-1}\sum_{%
\begin{array}{c}
n\equiv a(\func{mod}f) \\ 
n\geq 0%
\end{array}%
}^{\infty }\frac{q^{nh}\xi ^{n}}{(x+n)^{s-1}}.
\end{equation*}
\end{definition}

Relation between $\zeta _{\xi ,q}^{(h)}(s,x)$ and $H_{\xi ,q}^{(h)}(s,x+a:f)$
are given by%
\begin{equation}
H_{\xi ,q}^{(h)}(s,x+a:f)=\frac{q^{ha}\xi ^{a}}{f^{s}}\zeta _{\xi
^{f},q^{f}}^{(h)}(s,\frac{a+x}{f}).  \label{hq3}
\end{equation}%
Let $n\in \mathbb{Z}^{+}$. Substituting (\ref{aa7}) in the above and using (%
\ref{AA-1}), we obtain%
\begin{eqnarray*}
H_{\xi ,q}^{(h)}(1-n,x+a &:&f)=-\frac{B_{n,\xi ^{f}}^{(h)}(\frac{a+x}{f}%
,q^{f})}{n} \\
&=&-\frac{q^{ha}\xi ^{a}f^{n-1}}{n}\sum_{k=0}^{n}\left( 
\begin{array}{c}
n \\ 
k%
\end{array}%
\right) \left( \frac{x+a}{f}\right) ^{n-k}B_{k,\xi ^{f}}^{(h)}(\frac{x+a}{f}%
,q^{f}).
\end{eqnarray*}%
Thus, by the above equation, we obtain%
\begin{equation*}
H_{\xi ,q}^{(h)}(s,x+a:f)=\frac{a^{1-s}q^{ha}\xi ^{a}}{(s-1)f}%
\sum_{k=0}^{\infty }\left( 
\begin{array}{c}
1-s \\ 
k%
\end{array}%
\right) \left( \frac{f}{x+a}\right) ^{k}B_{k,\xi ^{f}}^{(h)}(q^{f}).
\end{equation*}%
By (\ref{aa5}) and (\ref{hq3}), we obtain the following relations:

Let $s$ be a complex variable, $a$ and $f$ be integers with $0<a<f$, $x\in 
\mathbb{R}$ with $0<x<1$, we have 
\begin{eqnarray}
L_{\xi ,q}^{(h)}(s,x,\chi ) &=&\sum_{a=1}^{f}\chi (a)H_{\xi
,q}^{(h)}(s,x+a:f)  \label{AA-2} \\
&=&\frac{1}{(s-1)f}\sum_{a=1}^{f}\chi (a)(x+a)^{1-s}q^{ha}\xi
^{a}\sum_{k=0}^{\infty }\left( 
\begin{array}{c}
1-s \\ 
k%
\end{array}%
\right) \left( \frac{f}{x+a}\right) ^{k}B_{k,\xi ^{f}}^{(h)}(q^{f}).  \notag
\end{eqnarray}%
By the above equation, $L_{\xi ,q}^{(h)}(s,x,\chi )$ is an analytic for $%
x\in \mathbb{R}$ with $0<x<1$ and $s\in \mathbb{C}$ except $s=1$.

\begin{remark}
Observe that if $\xi =1$, then $L_{\xi ,q}^{(h)}(s,x,\chi )$ is reduced to
the following equation cf. \cite{tkimnewApropqL-2006}: 
\begin{equation*}
L_{q}^{(h)}(s,x,\chi )=\frac{1}{(s-1)f}\sum_{a=1}^{f}\chi
(a)(x+a)^{1-s}q^{ha}\sum_{k=0}^{\infty }\left( 
\begin{array}{c}
1-s \\ 
k%
\end{array}%
\right) \left( \frac{f}{x+a}\right) ^{k}B_{k}^{(h)}(q^{f}).
\end{equation*}%
\ By (\ref{aa6a}), the values of $L_{\xi ,q}^{(h)}(s,x,\chi )$ at negative
integers are algebraic, hence may be regarded as lying in an extension of $%
\mathbb{Q}_{p}$. Consequently, we investigate a $p$-adic function which
agrees with at negative integers in the next section.
\end{remark}

Substituting $s=0$ into (\ref{AA-2}), we obtain%
\begin{equation*}
L_{\xi ,q}^{(h)}(0,x,\chi )=\frac{1}{f}\sum_{a=1}^{f}\chi (a)(x+a)q^{ha}\xi
^{a}\left( \frac{f\xi ^{f}q^{fh}\log q^{hf}-(x+a)(\xi ^{f}q^{fh}-1)\log
q^{fh}}{(x+a)(\xi ^{f}q^{fh}-1)^{2}}\right) .
\end{equation*}

\section{Twisted $p$-adic interpolation function for the $q$-extension of
the generalized Bernoulli polynomials}

In this section, we can use some notations which are due to Kim\cite%
{tkimnewApropqL-2006} and Washington\cite{L. C. Washington}. The integer $%
p^{\ast }$ is defined by $p^{\ast }=p$ if $p>2$ and $p^{\ast }=4$ if $p=2$
cf. (\cite{kimpqLkyus}, \cite{tkimnewApropqL-2006}, \cite{young}). Let $w$
denote the Teichm\"{u}ller character, having conductor $f_{w}=p^{\ast }$.
For an arbitrary character $\chi $, we define $\chi _{n}=\chi w^{-n}$, where 
$n\in \mathbb{Z}$, in the sense of the product of characters. In this
section, if $q\in \mathbb{C}_{p}$, then we assume $\mid 1-q\mid _{p}<p^{-%
\frac{1}{p-1}}$. Let \ $<a>=w^{-1}(a)a=\frac{a}{w(a)}$. We note that $%
<a>\equiv 1(\func{mod}p^{\ast }\mathbb{Z}_{p})$. Thus, we see that 
\begin{eqnarray*}
&<&a+p^{\ast }t>=w^{-1}(a+p^{\ast }t)(a+p^{\ast }t) \\
&=&w^{-1}(a)a+w^{-1}(a)(p^{\ast }t)\equiv 1(\func{mod}p^{\ast }\mathbb{Z}%
_{p}[t]),
\end{eqnarray*}%
where $t\in \mathbb{C}_{p}$ with$\mid t\mid _{p}\leq 1$, $(a,p)=1$. The $p$%
-adic logarithm function, $\log _{p}$, is the unique function $\mathbb{C}%
_{p}^{\times }\rightarrow \mathbb{C}_{p}$ that satisfies the following
conditions:

$i$)%
\begin{equation*}
\log _{p}(1+x)=\sum_{n=1}^{\infty }\frac{(-1)^{n}x^{n}}{n}\text{, }\mid
x\mid _{p}<1,
\end{equation*}

$ii$) $\log _{p}(xy)=\log _{p}x+\log _{p}y$, $\forall x,y\in \mathbb{C}%
_{p}^{\times }$,

$iii$) $\log _{p}p=0$.

Let 
\begin{equation*}
A_{j}(x)=\sum_{n=0}^{\infty }a_{n,j}x^{n},a_{n,j}\in \mathbb{C}_{p}\text{, }%
j=0,1,2,...
\end{equation*}%
be a sequence of power series, each of which converges in a fixed subset 
\begin{equation*}
D=\left\{ s\in \mathbb{C}_{p}:\mid s\mid _{p}\leq \mid p^{\ast }\mid
^{-1}p^{-\frac{1}{p-1}}\right\}
\end{equation*}%
of $\mathbb{C}_{p}$ such that

1) $a_{n,j}\rightarrow a_{n,0}$ as $j\rightarrow \infty $, for $\forall n$,

2) for each $s\in D$ and $\epsilon >0$, there exists $n_{0}=n_{0}(s,\epsilon
)$ such that $\mid \sum_{n\geq n_{0}}a_{n,j}s^{n}\mid _{p}<\epsilon $ for $%
\forall j$. Then $\lim_{j\rightarrow \infty }A_{j}(s)=A_{0}(s)$ for all $%
s\in D$. This is used by Washington\cite{L. C. Washington} to show that each
of the function $w^{-s}(a)a^{s}$ and 
\begin{equation*}
\sum_{k=0}^{\infty }\left( 
\begin{array}{c}
s \\ 
k%
\end{array}%
\right) \left( \frac{F}{a}\right) ^{k}B_{k}
\end{equation*}%
where $F$ is the multiple of $p^{\ast }$ and $f=f_{\chi }$, is analytic in $%
D $. We consider the \textit{twisted} $p$-adic analogs of the twisted two
variable $q$-$L$-functions, $L_{\xi ,q}^{(h)}(s,t,\chi )$. These functions
are the $q$-analogs of the $p$-adic interpolation functions for the
generalized twisted Bernoulli polynomials attached to $\chi $. Let $F$ be a
positive integral multiple of $p^{\ast }$ and $f=f_{\chi }$.

We define%
\begin{equation}
L_{\xi ,p,q}^{(h)}(s,t,\chi )=\frac{1}{(s-1)F}\sum_{%
\begin{array}{c}
a=1 \\ 
(a,p)=1%
\end{array}%
}^{F}\chi (a)<a+p^{\ast }t>^{1-s}q^{ha}\xi ^{a}\sum_{k=0}^{\infty }\left( 
\begin{array}{c}
1-s \\ 
k%
\end{array}%
\right) \left( \frac{F}{a+p^{\ast }t}\right) ^{k}B_{k,\xi ^{F}}^{(h)}(q^{F}).
\label{aLpq}
\end{equation}%
Then $L_{\xi ,p,q}^{(h)}(s,t,\chi )$ is analytic for $t\in \mathbb{C}_{p}$
with$\mid t\mid _{p}\leq 1$, provided $s\in D$, except $s=1$ when $\chi \neq
1$. For $t\in \mathbb{C}_{p}$ with$\mid t\mid _{p}\leq 1$, we see that%
\begin{equation*}
\sum_{k=0}^{\infty }\left( 
\begin{array}{c}
1-s \\ 
k%
\end{array}%
\right) \left( \frac{F}{a+p^{\ast }t}\right) ^{k}B_{k,\xi ^{F}}^{(h)}(q^{F})
\end{equation*}
is analytic for $s\in D$. By definition of $<a+p^{\ast }t>$, it is readily
follows that 
\begin{equation*}
<a+p^{\ast }t>^{s}=<a>^{s}\sum_{k=0}^{\infty }\left( 
\begin{array}{c}
s \\ 
k%
\end{array}%
\right) (a^{-1}p^{\ast }t)^{k}
\end{equation*}%
is analytic for $t\in \mathbb{C}_{p}$ with$\mid t\mid _{p}\leq 1$ when $s\in
D$. Since $(s-1)L_{\xi ,p,q}^{(h)}(s,t,\chi )$ is a finite sum of products
of these two functions, it must also be analytic for $t\in \mathbb{C}_{p}$
with$\mid t\mid _{p}\leq 1$, whenever $s\in D$.

Observe that%
\begin{equation*}
\lim_{s\rightarrow 1}(s-1)L_{\xi ,p,q}^{(h)}(s,t,\chi )=\frac{1}{F}\sum_{%
\begin{array}{c}
a=1 \\ 
(a,p)=1%
\end{array}%
}^{F}\chi (a)q^{ha}\xi ^{a}B_{0,\xi ^{F}}^{(h)}(q^{F}).
\end{equation*}%
Substituting $\chi =1$ in the above, then we have%
\begin{eqnarray*}
\lim_{s\rightarrow 1}(s-1)L_{\xi ,p,q}^{(h)}(s,t,\chi ) &=&\frac{B_{0,\xi
^{F}}^{(h)}(q^{F})}{F}\sum_{%
\begin{array}{c}
a=1 \\ 
(a,p)=1%
\end{array}%
}^{F}q^{ha}\xi ^{a} \\
&=&\frac{B_{0,\xi ^{F}}^{(h)}(q^{F})}{F}\left( \frac{1-q^{hF}\xi ^{F}}{%
1-q^{h}\xi }-\frac{1-q^{hpF}}{1-q^{ph}}\right) 
\end{eqnarray*}%
By definition of $B_{0,\xi ^{F}}^{(h)}(q^{F})$ in (\ref{ayb1}), we obtain%
\begin{eqnarray*}
\func{Re}z_{s=1}L_{\xi ,p,q}^{(h)}(s,t,\chi ) &=&\lim_{s\rightarrow
1}(s-1)L_{\xi ,p,q}^{(h)}(s,t,\chi ) \\
&=&\frac{\log q^{h}}{q^{h}\xi -1}\left( \frac{1-q^{hF}\xi ^{F}}{1-q^{h}\xi }-%
\frac{1-q^{hpF}}{1-q^{hp}}\right) ,
\end{eqnarray*}%
when $\chi =1$. Let $n\in \mathbb{Z}^{+}$ and $t\in \mathbb{C}_{p}$ with$%
\mid t\mid _{p}\leq 1$. Since $F$ must be a multiple of $f=f_{\chi _{n}}$,
by (\ref{aa6}), we obtain%
\begin{equation}
B_{n,\chi _{n},\xi }^{(h)}(p^{\ast }t,q)=F^{n-1}\sum_{a=0}^{F}\chi
_{n}(a)\xi ^{a}q^{ha}B_{n,\xi ^{F}}^{(h)}(\frac{a+p^{\ast }t}{F},q^{F}).
\label{hqp-4}
\end{equation}%
If $\chi _{n}(p)=0$, then $(p,f_{\chi _{n}})=1$, so that $\frac{F}{p}$ is a
multiple of $f_{\chi _{n}}$. Consequently, we get%
\begin{equation}
\chi _{n}(p)p^{n-1}B_{n,\chi _{n},1}^{(h)}(p^{-1}p^{\ast
}t,q^{p})=F^{n-1}\sum_{%
\begin{array}{c}
a=0 \\ 
p\mid a%
\end{array}%
}^{F}\chi _{n}(a)\xi ^{a}q^{ha}B_{n,\xi ^{F}}^{(h)}(\frac{a+p^{\ast }t}{F}%
,q^{F}).  \label{hqp-5}
\end{equation}%
The difference of (\ref{hqp-4}) and (\ref{hqp-5}), we have%
\begin{eqnarray*}
&&B_{n,\chi _{n},\xi }^{(h)}(p^{\ast }t,q)-\chi _{n}(p)p^{n-1}B_{n,\chi
_{n},1}^{(h)}(p^{-1}p^{\ast }t,q^{p}) \\
&=&F^{n-1}\sum_{%
\begin{array}{c}
a=0 \\ 
p\nshortmid a%
\end{array}%
}^{F}\chi _{n}(a)\xi ^{a}q^{ha}B_{n,\xi ^{F}}^{(h)}(\frac{a+p^{\ast }t}{F}%
,q^{F}).
\end{eqnarray*}%
By using (\ref{AA-1}), we obtain%
\begin{eqnarray*}
B_{n,\xi ^{F}}^{(h)}(\frac{a+p^{\ast }t}{F},q^{F}) &=&\sum_{k=0}^{n}\left( 
\begin{array}{c}
n \\ 
k%
\end{array}%
\right) \left( \frac{a+p^{\ast }t}{F}\right) ^{n-k}B_{k,\xi
^{F}}^{(h)}(q^{F}) \\
&=&(a+p^{\ast }t)^{n}F^{-n}\sum_{k=0}^{n}\left( 
\begin{array}{c}
n \\ 
k%
\end{array}%
\right) \left( \frac{F}{a+p^{\ast }t}\right) ^{k}B_{k,\xi ^{F}}^{(h)}(q^{F}).
\end{eqnarray*}%
Since $\chi _{n}(a)=\chi (a)w^{-n}(a)$, $(a,p)=1$, and $t\in \mathbb{C}_{p}$
with$\mid t\mid _{p}\leq 1$, we have%
\begin{eqnarray*}
&&B_{n,\chi _{n},\xi }^{(h)}(p^{\ast }t,q)-\chi _{n}(p)p^{n-1}B_{n,\chi
_{n},1}^{(h)}(p^{-1}p^{\ast }t,q^{p}) \\
&=&\frac{1}{F}\sum_{%
\begin{array}{c}
a=1 \\ 
(a,p)=1%
\end{array}%
}^{F}\chi (a)<a+p^{\ast }t>^{n}q^{ha}\xi ^{a}\sum_{k=0}^{\infty }\left( 
\begin{array}{c}
n \\ 
k%
\end{array}%
\right) \left( \frac{F}{a+p^{\ast }t}\right) ^{k}B_{k,\xi ^{F}}^{(h)}(q^{F}).
\end{eqnarray*}%
Substituting $s=1-n$, $n\in \mathbb{Z}^{+}$, into (\ref{aLpq}), we obtain%
\begin{equation*}
L_{\xi ,p,q}^{(h)}(1-n,t,\chi )=-\frac{B_{n,\chi _{n},\xi }^{(h)}(p^{\ast
}t,q)-\chi _{n}(p)p^{n-1}B_{n,\chi _{n},\xi ^{p}}^{(h)}(p^{-1}p^{\ast
}t,q^{p})}{n}.
\end{equation*}%
Consequently, we arrive at the following main theorem:

\begin{theorem}
Let $F$ be a positive integral multiple of $p^{\ast }$ and $f=f_{\chi _{n}}$%
, and let%
\begin{eqnarray*}
&&L_{\xi ,p,q}^{(h)}(s,t,\chi ) \\
&=&\frac{1}{(s-1)F}\sum_{%
\begin{array}{c}
a=1 \\ 
(a,p)=1%
\end{array}%
}^{F}\chi (a)<a+p^{\ast }t>^{1-s}q^{ha}\xi ^{a}\sum_{k=0}^{\infty }\left( 
\begin{array}{c}
1-s \\ 
k%
\end{array}%
\right) \left( \frac{F}{a+p^{\ast }t}\right) ^{k}B_{k,\xi ^{F}}^{(h)}(q^{F}).
\end{eqnarray*}%
Then $L_{\xi ,p,q}^{(h)}(s,t,\chi )$ is analytic for $h\in \mathbb{Z}^{+}$
and $t\in \mathbb{C}_{p}$ with$\mid t\mid _{p}\leq 1$, provided $s\in D$,
except $s=1$. Also, if $t\in \mathbb{C}_{p}$ with$\mid t\mid _{p}\leq 1$,
this function is analytic for $s\in D$ when $\chi \neq 1$, and meromorphic
for $s\in D$, with simple pole at $s=1$ having residue%
\begin{equation*}
\frac{\log q^{h}}{q^{h}\xi -1}\left( \frac{1-q^{hF}\xi ^{F}}{1-q^{h}\xi }-%
\frac{1-q^{hpF}}{1-q^{ph}}\right) 
\end{equation*}%
when $\chi =1$. In addition, for each $n\in \mathbb{Z}^{+}$, we have%
\begin{equation*}
L_{\xi ,p,q}^{(h)}(1-n,t,\chi )=-\frac{B_{n,\chi _{n},\xi }^{(h)}(p^{\ast
}t,q)-\chi _{n}(p)p^{n-1}B_{n,\chi _{n},\xi ^{p}}^{(h)}(p^{-1}p^{\ast
}t,q^{p})}{n}.
\end{equation*}
\end{theorem}

\begin{remark}
Observe that $\lim_{\xi \rightarrow 1}L_{\xi ,p,q}^{(h)}(s,t,\chi
)=L_{p,q}^{(h)}(s,t,\chi )$ cf. \cite{tkimnewApropqL-2006}. $%
\lim_{h\rightarrow 1}L_{p,q}^{(h)}(s,0,\chi )=L_{p,q}(s,\chi )$ cf. (\cite%
{kimpqLkyus}, \cite{kimpqldiscrmath}). $\lim_{q\rightarrow 1}L_{p,q}(s,\chi
)=L_{p}(s,\chi )$, cf. (\cite{J. Diamond}, \cite{B. Ferrero and R. Greenberg}%
, \cite{K. Iwasawa}, \cite{N. Koblitz}, \cite{N. Koblitz2}, \cite{fox}, \cite%
{Shratani and S. Yamamoto}, \cite{L. C. Washington}, \cite{simsekRJMP}).
\end{remark}

We defined Witt's formula for $B_{n,\chi ,\xi }^{(h)}(z,q)$ polynomials as
follows:%
\begin{equation*}
B_{n,\chi ,\xi }^{(h)}(q)=\int_{\mathbb{X}}\chi (x)\xi ^{x}q^{hx}x^{n}d\mu
_{1}(x)\text{, cf. (\cite{ysimsek}, \cite{simsekCanada})}
\end{equation*}%
where $\mid 1-q\mid _{p}\leq p^{-\frac{1}{p-1}}$.

By using this formula, we define%
\begin{equation*}
L_{\xi ,p,q}^{(h)}(s,\chi )=\frac{1}{s-1}\int_{X^{\ast }}\chi (x)\xi
^{x}<x>^{-s}q^{hx}d\mu _{q}(x),
\end{equation*}%
where $s\in \mathbb{C}_{p}$.

Substituting $s=1-n$, $n\in \mathbb{Z}^{+}$, into the above, we have%
\begin{eqnarray*}
L_{\xi ,p,q}^{(h)}(1-n,\chi ) &=&-\frac{1}{n}\int_{X^{\ast }}\xi ^{x}\chi
(x)<x>^{n-1}q^{hx}d\mu _{q}(x) \\
&=&-\frac{1}{n}\left( \int_{\mathbb{X}}\chi _{n}(x)\xi ^{x}q^{hx}x^{n}d\mu
_{q}(x)-\int_{p\mathbb{X}}\chi _{n}(px)\xi ^{px}q^{phx}x^{n}d\mu
_{q}(px)\right) \\
&=&-\frac{B_{n,\chi ,\xi }^{(h)}(q)-\chi _{n}(p)p^{n-1}B_{n,\chi ,\xi
^{p}}^{(h)}(q^{p})}{n}.
\end{eqnarray*}

Consequently, we arrive at the following theorem:

\begin{theorem}
Let $s\in \mathbb{C}_{p}$ and let%
\begin{equation*}
L_{\xi ,p,q}^{(h)}(s,\chi )=\frac{1}{s-1}\int_{X^{\ast }}\chi (x)\xi
^{x}<x>^{-s}q^{hx}d\mu _{q}(x).
\end{equation*}%
For $n\in \mathbb{Z}^{+}$, we have%
\begin{eqnarray*}
L_{\xi ,p,q}^{(h)}(1-n,\chi ) &=&-\frac{1}{n}\int_{X^{\ast }}\xi ^{x}\chi
(x)<x>^{n-1}q^{hx}d\mu _{q}(x) \\
&=&-\frac{B_{n,\chi ,\xi }^{(h)}(q)-\chi _{n}(p)p^{n-1}B_{n,\chi ,\xi
^{p}}^{(h)}(q^{p})}{n}.
\end{eqnarray*}
\end{theorem}

\begin{acknowledgement}
This paper was supported by the Scientific Research Project Administration
Akdeniz University.
\end{acknowledgement}

\end{document}